\documentclass[11pt]{amsart}

\usepackage{xspace}

\usepackage{amsmath}
\usepackage{amsthm}
\usepackage{amssymb}

 \usepackage{url}
 \usepackage{graphicx} 
 \usepackage[all]{xy} 

\allowdisplaybreaks 

\newtheorem{theorem}{Theorem}[subsection]
\newtheorem{proposition}[theorem]{Proposition}
\newtheorem{lemma}[theorem]{Lemma}
\newtheorem{corollary}[theorem]{Corollary}

\theoremstyle{definition}
\newtheorem{definition}[theorem]{Definition}
\newtheorem{example}[theorem]{Example}

\theoremstyle{remark}
\newtheorem{remark}{Remark}

\numberwithin{figure}{section}

\newcommand{\cat}{\mathop{\mathrm{cat}}}

\newcommand{\sd}{\mathop{\mathrm{sd}}}

\newcommand{\eqdef}{\mathrel{\mathop:}=}

\newcommand{\geo}[1]{\vert #1 \vert}

\newcommand{\id}{\mathop{\mathrm{id}}}

\newcommand{\scat}{\mathop{\mathrm{scat}}}
\newcommand{\gscat}{\mathop{\mathrm{gscat}}}
\newcommand{\wcat}{\mathrm{cat}^{\mathrm{Wh}}}
\newcommand{\wscat}{\mathrm{scat}^{\mathrm{Wh}}}

\newcommand{\mesh}{\mathop{\mathrm{mesh}}}
\newcommand{\st}{\mathrm{st}}

\newcommand{\lk}{\mathrm{lk}}

\newcommand{\hueco}{\begin{matrix}&&\cr&&\cr &&\cr&&\cr \end{matrix}}

\newcommand{\figref}[1]{Figure~\ref{#1}}

\title[Simplicial LS category]{Simplicial Lusternik-Schnirelmann category}

\author{D.~Fern\'andez-Ternero \and E.~Mac\'{\i}as-Virg\'os \and E.~Minuz \and J.A.~Vilches}

\thanks{The first and the fourth authors were partially supported by MINECO Spain Research Project MTM2015-65397-P and Junta de Andaluc\'{\i}a Research Groups FQM-326 and FQM-189. The second  author was partially supported by MINECO Spain Research Project MTM2013-41768-P and FEDER}

\date{}

\begin{document}

\begin{abstract}
The simplicial LS-category of a finite abstract simplicial
complex is a new invariant of the strong homotopy type, defined in
purely combinatorial terms. It generalizes to arbitrary simplicial
complexes the well known notion of arboricity of a graph, and allows
to develop all the machinery of algebraic topology which is costumary
in the classical theory of Lusternik-Schnirelmann category.

\end{abstract}



\maketitle

\setcounter{tocdepth}{1}
\tableofcontents

\section{Introduction}
In a previous paper \cite{FMV2015} we introduced the so-called
simplicial Lus\-ter\-nik-Schnirelmann category of a simplicial
complex. This invariant, denoted by $\scat K$, is defined directly
from the combinatorial structure of the complex $K$, instead of considering
the topological LS-category of the geometric realization
$\geo{K}$, by replacing the
notion of homotopy by that of contiguity. With this approach it turned out that the simplicial
LS-category is an invariant of the strong homotopy type, as
defined by Barmak and Minian in \cite{BARMAKMINIAN2012}. Also,
simplicial LS-category is closely related to the LS-category of
the finite $T_0$-space represented by the Hasse diagram of $K$
\cite{MCCORD}. It is worth noting that although the idea of
contiguous simplicial maps is a classic one, the corresponding
theory of Lusternik-Schnirelmann category  had not been developed
until now. Then, many natural questions remained unsolved in this
new context. The aim of this paper is to answer some of them: for
instance, the relationship between $\scat K$ and $\cat \geo{K}$,
the Whitehead formulation of simplicial category, or an analogue
of Varadarajan's formula for fibrations, among others.  Other
results include a direct proof of the formula $\scat \sd K
\leq\scat K$, for the barycentric subdivision, or  a counterexample
of the analogue of the homotopy extension property for
subcomplexes. A final important result will be that simplicial
category generalizes to arbitrary simplicial complexes the well
known notion of arboricity of a graph.

The contents of the paper are as follows.

We start (Section \ref{2}) by introducing the definition of simplicial
Lusternik-Schnirel\-mann category $\scat K$ of the simplicial
complex $K$. The idea is to define an analogue of the classical
LS-category of a topological space, by replacing the notion of
homotopic maps by that of contiguous simplicial maps. We then
prove some basic properties of this new notion of category, and we
recall our previous result \cite{FMV2015} that $\scat$ is an
invariant of the strong homotopy type. In particular, the category
of the simplicial complex $K$ equals that of its core $K_0$ (the
minimal subcomplex of $K$ obtained by strong collapses). As a
corollary we prove that $\scat K$ is bounded above by the number
of vertices and the number of maximal simplices of $K_0$.

In Section \ref{3} we study the effect of barycentric subdivision
on simplicial category and we give a direct proof of $\scat \sd K
\leq \scat K$. For that we need to prove that if $\varphi$, $\psi$
are contiguous maps, then the induced maps $\sd \varphi$, $\sd
\psi$, are in the same contiguity class. The only proof that we
found in the literature is as a consequence of the results given in \cite{BARMAKMINIAN2012} via
finite posets.

In Section \ref{4} we compare the simplicial category of the
complex $K$ with the classical LS-category of the geometric
realization $\geo{K}$, proving that  $\cat \geo{K}\leq \scat K$.
This implies that $\cat \geo{K}\leq \scat \sd^N\!K $, for any
iterated barycentric subdivision $\sd^N\!K$ of $K$, although there
are examples where this inequality is strict (we show one of these
kind of examples).


Section \ref{5} deals with the study of simplicial products. Being more precise, we consider the
categorical product, denoted by $K\times L$. Then, we prove that
$\scat(K\times L)+1\leq (\scat K+1)(\scat L+1)$.


We are now (Section \ref{6}) in a position to discuss the
simplicial analogue of the Whitehead formulation of the
LS-category. It is known that for topological spaces $X$ with
``good'' properties the LS-category can be computed as the least
integer $n=\wcat X$ such that  the diagonal map   $\Delta \colon X
\to X^{n+1}$ factors (up to homotopy) through the so-called  fat
wedge $T^{n+1}X$.  We  try to adapt this result to abstract
simplicial complexes. First, we define the $n$-th fat wedge $T^nK$
for any pointed complex $K$ and we study its behaviour under
contiguity. We then define a  simplicial Whitehead category
$\wscat K$, which is an invariant of the pointed strong homotopy
type and we are able to prove that $\scat K\leq \wscat K$.
However, unlike the continuous case, the other inequality is not
true, as we show by an example.

In Section \ref{7}  we show how the partial failure of the
Whitehead formulation  is related to the contiguity extension
property.  We discuss a new notion of cofibration in the
simplicial setting.  It is well known   that if $A$ is a
subcomplex of a CW-complex $X$ then the pair $(X,A)$ has the
homotopy extension property. Therefore, if $L$ is a
subcomplex of a simplicial complex $K$,  the pair $(\geo{K},\geo{L})$
has the (topological) homotopy extension property. We define a
purely combinatorial analogue for a simplicial pair $(K,L)$, but
we show counterexamples where the contiguity extension property
fails to be true.

Finally, Section \ref{8} is focused on the study of the simplicial
LS-category in the one-dimensional case, that is, on graphs. The
well known graph-theoretical notion of {\em arboricity} will play
a central role in this study. The arboricity $\Upsilon(G)$ of the
graph $G$ is  the cardinality of a minimal decomposition of $G$
into disjoint spanning forests, i.e., acyclic subgraphs which are
non-necessarily connected and cover all the vertices. The aim of
this section is to prove that arboricity coincides with simplicial
category, being more precise, we prove that $\Upsilon(G)=\scat
G+1$. The two main ideas are that categorical subcomplexes are
precisely forests and that the notion of contiguity is a very
rigid one in the setting of graphs, allowing only a limited number
of moves.

As a final comment, we emphasize that simplicial LS-category is a
new strong homotopy invariant, defined in purely combinatorial
terms, that generalizes to arbitrary simplicial complexes the well
known notion of arboricity of a graph,  and that allows to develop
all the machinery of algebraic topology which is costumary in the
classical theory of Lusternik-Schnirelmann category.


\section{Simplicial category}\label{2}
\subsection{}
We recall the notion of  simplicial LS-category, introduced by the auhors in  \cite{FMV2015}.

Let $K$ and $L$ be abstract simplicial complexes. Remember (see
\cite[\textsection 3.5]{SPANIER1966})
that two simplicial maps
$\varphi,\psi\colon  K\to L$ are {\em contiguous}, denoted by
$\varphi\sim_c \psi$, if for any simplex
$\sigma=\{v_0,\dots,v_p\}$ of $K$, the set of vertices
$$\varphi(\sigma)\cup\psi(\sigma)=\{\varphi(v_0),\dots,\varphi(v_p), \psi(v_0),\dots,\psi(v_p)\}$$ is a simplex of $L$.

More generally, two simplicial maps $\varphi,\psi\colon K \to L$ are in the same {\em contiguity class}, denoted by $\varphi\sim  \psi$, if there is a sequence  of simplicial maps $\varphi_i\colon K\to L$, $i=1,\dots,m$, such that
$$\varphi=\varphi_1\sim_c  \varphi_2\sim_c \cdots\sim_c   \varphi_m=\psi.$$

\begin{proposition}\label{FULL}
Let $\varphi,\psi\colon K \to L$ be  two simplicial maps. Let $L^\prime\subset L$ be a full subcomplexes such that $\varphi(K^\prime), \psi(K^\prime)\subset L^\prime$.  If $\varphi$ and $\psi$ are contiguous then their restrictions $\varphi^\prime,\psi^\prime\colon K^\prime \to L^\prime$ are contiguous too.
\end{proposition}

The concept of contiguity class  provides a simplicial analogue of   homotopy classes of continuous maps between  topological spaces.

By means of the notion of contiguity we define the key concept of
categorical subcomplex.

\begin{definition}
Let $K$ be an abstract simplicial complex.   We say that the subcomplex $U\subset K$ is {\em categorical} (in $K$) if there exists a vertex $v\in K$ such that the inclusion $i=i_U\colon U \to K$
 and the constant map $c=c_v\colon U \to K$ are in the same contiguity class (denoted by $i\sim c$).
\end{definition}

In other words, the inclusion of $U$ factors through the vertex $v\in K$ up to ``simplicial homotopy'' (meaning contiguity class). Notice that a categorical subcomplex may not  be connected.

\begin{definition}\label{RELATIVE}Let $K$ be an abstract simplicial complex. The {\em simplicial LS-category} of $K$, denoted by $\scat K$,  is the least integer $n\geq 0$ such that there exist
$n+1$ categorical subcomplexes $U_0,\dots,U_n$ of $K$ which cover $K$, that is, such that $K=  U_0\cup\cdots\cup U_n$.
 \end{definition}

For instance,  as we shall explain in subsection \ref{PROPERTIES}, $\scat K=0$  if and only if $K$ is strongly collapsible to a point, in the sense given by Barmak and Minian  \cite{BARMAK2011, BARMAKMINIAN2012}.

\subsection{}We now show that the simplicial LS-category can be computed by taking into account only maximal simplices. The proof is inspired on an idea from J.~Strom for finite topological spaces.

\begin{lemma}\label{SUBCAT}Let $V\subset U\subset K$ be subcomplexes of $K$. If $U$ is categorical (in $K$) then $V$ is categorical (in $K$).
\end{lemma}
\begin{proof}
If the inclusion $i\colon U \subset K$ verifies $i\sim c$ for some
constant map $c\colon U\to K$, and $j\colon V\subset U$ is the
inclusion, then $i\circ j\colon V\to K$ is the inclusion, $c\circ
j\colon V\to K$ is a constant map, and $c\circ j\sim i\circ j$.
\end{proof}

\begin{proposition}\label{MAX} In Definition \ref{RELATIVE} one may assume that:
\begin{enumerate}
\item
each  categorical subcomplex in the covering  $U_0,\dots,U_n$ is a union of maximal simplices of $K$;
\item
each maximal simplex of $K$ is contained in only one of the elements of the covering.
\end{enumerate}
 \end{proposition}

\begin{proof}
Let $U_0,\dots,U_n\subset K$ be a categorical covering. We shall replace each subcomplex $U_j$ by another subcomplex $V_j$ (may be empty) defined as
the union of the simplices $\sigma\in U_j$ which are maximal in $K$.

First, $V_0,\dots,V_n$ is a covering of $K$, because if $\mu$ is a simplex of $K$, it must be contained in some maximal simplex $\sigma$, which in turns is contained in some $U_j$. Then $\mu\subset \sigma\in V_j$, so $\mu\in V_j$.

Moreover, $V_j$ is categorical (by Lemma \ref{SUBCAT}), because $V_j\subset U_j$ .

The second part follows from the fact that if we suppress each maximal simplex from all excepting one of the $V_j$, then, the resulting subcomplexes are still categorical, by Lemma \ref{SUBCAT}, and they cover $K$.
\end{proof}


\subsection{}\label{PROPERTIES}
We state several general properties of simplicial LS-category.

In  \cite{BARMAKMINIAN2012}, see also \cite{BARMAK2011}, Barmak
and Minian introduced the notion of {\em strong collapse}, a
particular type of simple collapse which is specially adapted to
the simplicial structure. Actually, it can be modelled as a
simplicial map, in contrast with the standard concept of collapse.

\begin{definition}A vertex $v$ of a simplicial complex $K$ is {\em dominated} by another vertex $v^\prime$ if every maximal simplex that contains $v$ also contains $v^\prime$. Equivalently, the link of $v$
is a simplicial cone with vertex $v^\prime$. \end{definition}

\begin{figure}[h!]
\begin{center}
 \includegraphics[height=30mm]{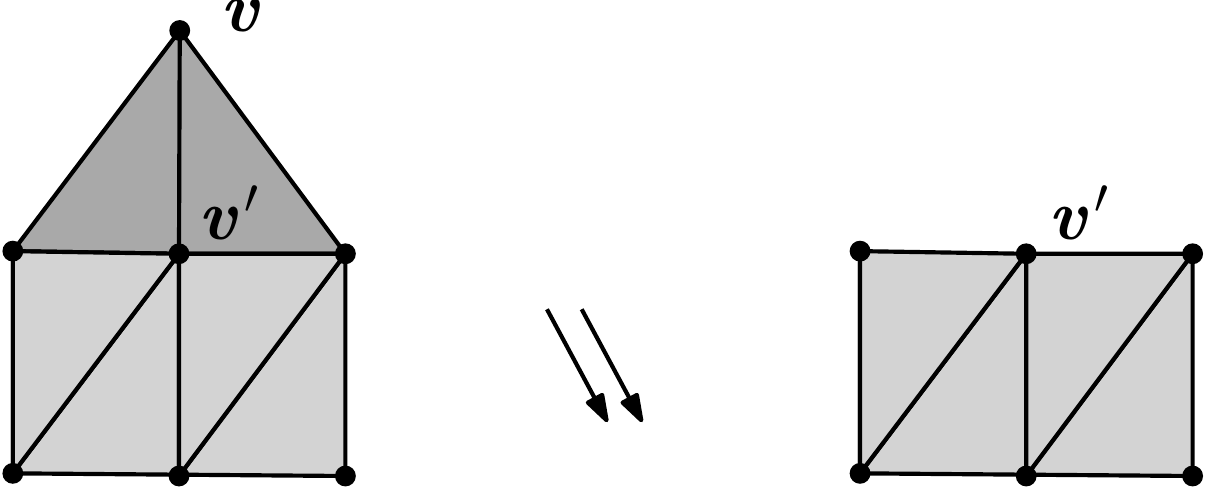}
\end{center}
 \end{figure}

An {\em elementary strong collapse} consists of removing the open
star of a dominated vertex $v$ from a simplicial complex $K$.  The
inverse of a strong collapse is called a strong expansion. Then,
two simplicial complexes $K,L$ have the same {\em strong homotopy
type}, denoted by $K\sim L$, if they are related by a sequence of
strong collapses and expansions. Surprisingly, this turns out to
be intimately related to the classical notion of contiguity. More
precisely, having the same strong homotopy type is equivalent  to
the existence of simplicial maps $\varphi\colon K\to L$ and
$\psi\colon L \to K$ such that
 $\psi\circ\varphi\sim 1_K$ and $\varphi\circ\psi\sim 1_L$ \cite[Corollary 2.12]{BARMAKMINIAN2012}.   The strong homotopy type gives a simplicial analogue to the homotopy type of topological spaces.

Notice that $\scat K=0$  if and only if $K$ is {\em strongly
collapsible}, that is, there is a finite sequence
of elementary strong collapses reducing it to a vertex.

\begin{example}
The simplicial complex $K=\{0,1,2\}$ in Figure \ref{examplecone} is not strong collapsible, in fact $\scat K=1$.
\begin{figure}[h!]
\begin{center}
\includegraphics[height=25mm]{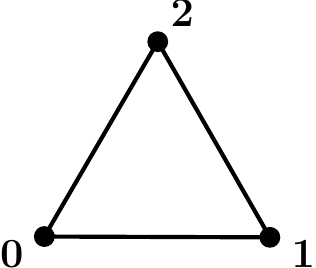}
\end{center}
\caption{A complex $K$ with $\scat K=1$}
\label{examplecone}
 \end{figure}
 \end{example}

\begin{example}
Let $K$ be a finite simplicial complex. The cone over $K$, $K\ast
a$ is strong collapsible. That is, $\scat(K\ast a)=0$.

This result can be
proven  just by taking into
account that $\lk(u,K\ast a)=\lk(u,K)\ast a$ \cite{AYALA1987}.
\end{example}

\begin{example}
Let $K$ be a finite simplicial complex such that $\scat K >0$, and
$S^{0}$ the complex given by only two $0$- simplices and no other
simplices. We define the suspension of $K$ as $\Sigma K=K\ast
S^{0}$. Then $\scat(\Sigma K) =1$.
\end{example}

Next theorem  was proved in \cite[Theorem 3.4]{FMV2015}.
\begin{theorem} The simplicial LS-category is an invariant of the strong homotopy type, that is,  $K\sim L$ implies $\scat K=\scat L$.
\end{theorem}
\label{SCORE} Therefore a simplicial complex $K$ and its  {\em
core} $K_0$ have the same simplicial LS-category. Remember that
the core of a complex is the minimal subcomplex obtained by
eliminating dominated vertices  (see for instance Barmak's book
\cite{BARMAK2011}).  More precisely, if a vertex $v$ is dominated
by another vertex $v^\prime$, then the collapse $r\colon K \to
K\setminus v$ is a strong equivalence. Under a finite number of
steps one attains a complex $K_0$ which is minimal, that is, it
does not have dominated vertices. This minimal complex  is unique
up to simplicial isomorphism.

Next result establishes two combinatorial upper bounds for the
simplicial category which do not exist for the usual LS-category.
\begin{corollary}\label{BOUNDS}
$\scat K$  is strictly bounded from above by the number of vertices and the number of maximal simplices  of its core $K_0$.
\end{corollary}
\begin{proof}
The star of a vertex  $\st(v)$, $v\in K_0$, is a strong
collapsible subcomplex of $K_0$ because all vertices in $\st(v)$
are dominated by $v$. Therefore, the family $\{\st(v)\colon v\in
K_{0}\}$ provides a cover of $K_0$ by categorical subcomplexes,
and so we have that $\scat K_0< m$, where $m$ is the number of
vertices of $K_0$. Since the value of $\scat K$ is a strong
homotopy invariant,  $\scat K =\scat K_{0}<m$.

On the other hand, let $M(K_0)$ be the number of maximal simplices
of $K_0$. Then, from Proposition \ref{MAX} it follows that $\scat
K_0<M(K_0)$ because  any maximal simplex  is a strong collapsible
subcomplex of $K_0$.
\end{proof}


\subsection{}\label{GEOMCAT} In \cite{FMV2015}, the authors also introduced the notion of {\em geometric} simplicial category $\gscat K$ of the complex $K$, which is   the analogue of the
geometric LS-category of a topological space introduced by Fox in
\cite{FOX1941}, see also \cite[\textsection 3.1]{CLOT2003}. The
difference with $\scat K$ is that each subcomplex in a categorical
covering is required to be strongly collapsible  in itself, rather
than in the ambient complex; in other words,  the identity, rather
than the inclusion, is in the contiguity class of a constant map.
Clearly,  $\scat K \leq \gscat K$.  Geometric simplicial category
has a very different behaviour from that of the ordinary category.
In particular it is neither hereditary nor homotopically
invariant. However, the bounds of Corollary \ref{BOUNDS} are still
true for $\gscat K$ because $\gscat K_{0}=\max\{\gscat L \colon
L\sim K\}$, as proven in \cite{FMV2015}.

\begin{definition}
Let $K$ and $L$ be two finite simplicial complexes. We define the
join $K\ast L$ as  the simplicial complex with set of  vertices
$K^{0}\cup L^{0}$ and with simplices the simplices of $K$, the
simplices of $L$ and the simplices given by $\sigma\cup\tau$,
$\sigma$  simplex in K and $ \tau$ simplex in L.
\end{definition}

\begin{proposition}
Let $K$ and $L$ be two finite simplicial complexes, then
$\gscat(K\ast L)\leq \min\{\gscat K, \gscat L\}$.
\end{proposition}

\begin{proof}
Suppose that $\min\{\gscat K, \gscat L\}=\gscat K=n$, then there
are $n+1$ strong collapsible  subcomplexes $U_{0},..,U_{n}$.
Consider the subcomplexes $U_{0}\ast L,..,U_{n}\ast L$ covering
$K$. They are strong collapsible and they provide a cover of
$K\ast L$.
\end{proof}


\section{Barycentric Subdivision}\label{3}
\subsection{}

We now study the behaviour of $\scat$ under barycentric
subdivisions. Our main result states that $\scat$ is decreasing
under such kind of subdivisions.

\begin{theorem}\label{SUBDIV}Let $\sd K$ be the first barycentric subdivision of $K$. Then $\scat (\sd K) \leq \scat K$.
\end{theorem}

This theorem was proved in \cite[Cor. 6.7]{FMV2015} using results
of Barmak and Minian (precisely, \cite[Prop. 4.11 and Prop.
4.12]{BARMAKMINIAN2012}) about finite spaces. We shall reformulate
them in order to give a direct proof.

If $K$ is an abstract simplicial complex, its first barycentric subdivision  can be defined formally as the complex $\sd K$ whose vertices $\{\sigma\}$ are identified to the simplices $\sigma=\{v_0,\dots,v_p\}$ of $K$, while the simplices of $\sd K$ are the sequences $\{\sigma_1,\dots,\sigma_q\}$  of simplices of $K$ such that $\sigma_1\subset\dots\subset\sigma_q$ (see \cite[\textsection 2.1]{KOZLOV2008}).

\begin{definition}Let $\varphi\colon K \to L$ be a simplicial map. The induced map $\sd \varphi\colon \sd K \to \sd L$ is defined as
$$(\sd \varphi)(\{\sigma_1,\dots,\sigma_q\})=\{\varphi(\sigma_1),\dots,\varphi(\sigma_q)\}.$$
\end{definition}
Clearly $\sd\varphi$ is a simplicial map, $\sd\id=\id$ and $\sd(\varphi\circ\psi)=(\sd\varphi)\circ (\sd\psi)$.

\begin{proposition}\label{SDCONT}If the simplicial maps $\varphi,\psi\colon K \to L$ are  in the same contiguity class, $\varphi\sim\psi$, then $\sd\varphi,\sd\psi\colon\sd K \to \sd L$ are in the same
contiguity class, $\sd \varphi\sim\sd\psi$.
\end{proposition}

\begin{proof}Without loss of generality, we may assume that the maps $\varphi$ and $\psi$ are contiguous, $\varphi\sim_c\psi$, which means that $\varphi(\sigma)\cup\psi(\sigma)$  is a simplex, for any simplex $\sigma\in K$.
Let $F\colon\sd K \to \sd L$ be the map given by
$$F(\{\sigma_1,\dots,\sigma_q\})= \{\varphi(\sigma_1)\cup\psi(\sigma_1),\dots,\varphi(\sigma_q)\cup\psi(\sigma_q)\}.$$
We shall prove that $\sd\varphi\sim F$ by increasing step by step the size of the set
$$\Omega(\sd\varphi,F)=\{\sigma\in K\colon  (\sd \varphi)(\{\sigma\})=F(\{\sigma\})\}.$$
Note that $(\sd\varphi)(\{\sigma\})=\{\varphi(\sigma)\}$, while $F(\{\sigma\})=\{\varphi(\sigma)\cup\psi(\sigma)\}$.

If  $\sd\varphi=F$   there is nothing to prove. Otherwise, there
exists $\mu\in K$ such that   $\mu\notin\Omega(\sd\varphi,F)$, or
equivalently, $\varphi(\mu)$ is strictly contained in $
\varphi(\mu)\cup\psi(\mu)$. Let us take $\mu$ to be of maximal
dimension with this property; in this way
$\varphi(\sigma)=\varphi(\sigma)\cup\psi(\sigma)$ when $\mu$ is a
proper face of $\sigma$. Now, we can define a new map $F_1\colon
\sd K \to \sd L$ as
$$F_1(\{\sigma\})=
\begin{cases}
(\sd\varphi)(\{\sigma\}) &\mbox{if } \sigma\neq \mu,\\
F(\{\mu\}) &\mbox{if } \sigma=\mu.
\end{cases}$$
It follows:

(1) The map $F_1$ is simplicial. In fact,   if
$\{\sigma_1,\dots,\sigma_q\}$ is a simplex of $\sd K$, then
$$
F_1(\{\sigma_1,\dots,\sigma_q\})=\{\varphi(\sigma_1),\dots,\varphi(\sigma_q)\},$$
if $\sigma_j\neq \mu \mbox{ for all }j=1,\dots,q,$
while
$$
F_1(\{\sigma_1,\dots,\sigma_q\})=
\{\varphi(\sigma_1),\dots,\varphi(\mu)\cup\psi(\mu\},\dots ,\varphi(\sigma_q)\},$$
if $\sigma_j=\mu \mbox{ for some\ } j.$
In both cases the image is a   simplex of $\sd L$.  Note that in the
second case, by the maximality of $\mu$ cited above, it follows
that $\varphi(\mu)\cup\psi(\mu)\subset \varphi(\sigma_i)$ if
$i>j$.

(2) We have $\sd\varphi\sim_c F_1$, because if
$\{\sigma_1,\dots,\sigma_q\}\in \sd K$ then
\begin{align*}
&(\sd \varphi)(\{\sigma_1,\dots,\sigma_q\})\cup F_1(\{\sigma_1,\dots,\sigma_q\})\\
 =&
\{\varphi(\sigma_1),\dots,\varphi(\sigma_q)\}\cup\{F_1(\{\sigma_1\}),\dots,F_1(\{\sigma_q\})\},
\end{align*}
which equals the simplex
$$\{\varphi(\sigma_1),\dots,\varphi(\sigma_{j-1}),\varphi(\sigma_j)\cup\psi(\sigma_j),\dots,\varphi(\sigma_q)\cup\psi(\sigma_q)\},$$ where $j$ is the lowest index such that $\sigma_j=\mu$, if such a $j$ exists.

(3) Finally, $\Omega(\sd\varphi,F)\varsubsetneq\Omega(F_1,F)$,
by the definition of $F_1$.

By repeating this construction we shall obtain a sequence of contiguous maps
$$\sd\varphi\sim_c F_1 \sim_c\cdots\sim_c F,$$
which shows that $\sd\varphi\sim F$. Using the same argument for
$\psi$, we can prove that $\sd\psi\sim F$. Thus $\sd\varphi\sim
\sd\psi$, as claimed.
\end{proof}

\begin{remark}
However, subdivision does not preserve the strong homotopy type, as shown by the following example taken from \cite[Example 5.1.13]{BARMAK2011}. Consider $K$ to be the boundary of a $2$-simplex and $\sd K$ its barycentric subdivision as in Figure \ref{ERICA3}. They are both minimal complexes because they have no dominated vertices,  but they are not isomorphic, therefore they do not have the same strong homotopy type (see Lemma \ref{MINID}).

\begin{figure}[h!]
\begin{center}
 \includegraphics[height=25mm]{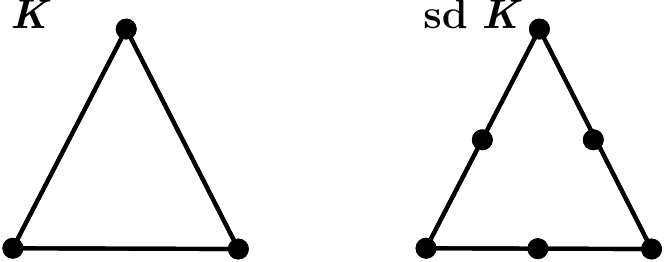}
\end{center}
\caption{A simplicial complex $K$ that has not the same strong homotopy type of its subdivision $\sd K$.}
\label{ERICA3}
 \end{figure}
 \end{remark}

\begin{proof}[Proof of Theorem \ref{SUBDIV}]
Let $\scat K=n$, and take a categorical covering $U_0,\dots,U_n$  of $K$. Consider the subcomplexes $\sd U_0,\dots,\sd U_n$, which cover $\sd K$. Since each inclusion $I_j\colon U_j \subset K$ is in the contiguity class of some constant map $v_j\colon U_j \to K$, denoted by $I_j\sim v_j$, it follows from Proposition \ref{SDCONT} that $\sd I_j\sim \sd v_j$. But it is clear that $\sd I_j$ is the inclusion  $\sd U_j\subset \sd K$, while $\sd v_j$ is the constant map $\{v_j\}$. Then each $\sd U_j$ is a categorical subcomplex and $\scat (\sd K) \leq n$.
\end{proof}

\begin{corollary}\label{SD1}
Let $K$ be a finite simplicial complex and let $\sd K$ be the first barycentric subdivision. Then, $\scat K=1$ implies that $\scat(\sd K)=1$.
\end{corollary}
\begin{proof}
In \cite[Theorem 4.15]{BARMAKMINIAN2012} it is proved that a complex  $K$ is strongly collapsible  if and only if  $\sd K$ is strongly collapsible. In other words, $\scat K=0$ if and only if $\scat(\sd K=0)$. Jointly with our Theorem \ref{SUBDIV} this ends the proof.
\end{proof}

\begin{example}\label{SUBDIVGRAPH} The following example, suggested to the third author by
J. Barmak, shows a complex where the inequality of Proposition
\ref{SUBDIV} is strict.

Let $K$ be the complete graph $K_{5}$ (see Figure \ref{K5})
considered as a $1$-dimensional simplicial complex. Let us consider the following categorical cover of $K$:
 \begin{align*}
U_0&= v_0v_1\cup v_0v_2 \cup v_0v_3\cup v_0v_4,\\
U_1&=v_1v_4\cup v_1v_2\cup v_2v_3,\\
U_2&=v_1v_3\cup v_3v_4\cup v_2v_4.
 \end{align*}

\begin{figure}[h!]
\begin{center}
 \includegraphics[height=50mm]{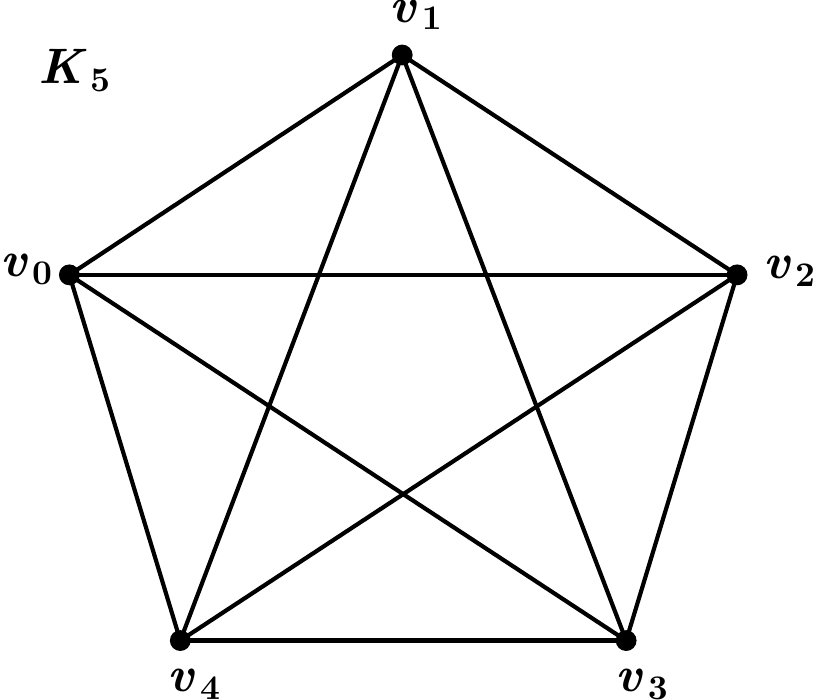}
\end{center}
\caption{The graph $K=K_5$ verifies $\scat(\sd K)<\scat
K$.} \label{K5}
 \end{figure}

Therefore $\scat K_5\leq 2$. Moreover, there is no cover of two
categorical subcomplexes. In fact, if it happens, then one of the
two subcomplexes has to contain at least $5$ edges. Moreover, since any tree contains one more vertex than the number of edges, a forest contains more vertices than edges. Hence, any forest with $5$ edges should have at least $6$ vertices, but it is impossible in our complex.

We conclude that there is not any categorical subcomplex with at
least five vertices, therefore there is no covering of $K_{5}$
given by two categorical subcomplexes. Hence we have $\scat K =2$.

On the other hand, the first barycentric subdivision $\sd K$ has a
covering with two categorical subcomplexes given, for example, by
the two subcomplexes that are showed in Figure
\ref{SDK5}: $L_{0}$, drawn with continuous edges , and
$L_{1}$, drawn with dashed edges. Since $\sd K_{5}$ is not strongly collapsible we can
conclude that $1=\scat(\sd K_{5})<\scat K_{5}=2$.

\begin{figure}[h!]
\begin{center}
\includegraphics[height=50mm]{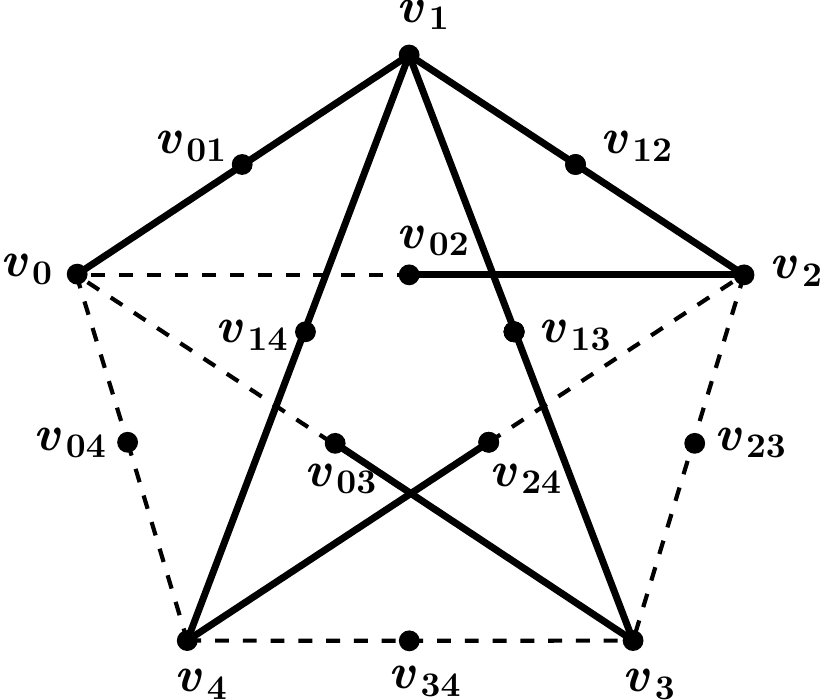}
\end{center}
\caption{A covering of $\sd K_{5}$ with two categorical
subcomplexes.}
 \label{SDK5}
 \end{figure}
\end{example}

\begin{remark}

Note that $\scat K_5$ equals the {\em arboricity} of $K_5$ minus
one; later we shall prove that this is a general result for any
graph (Theorem \ref{ARBCAT}). Also, we shall show later that for graphs
the simplicial category equals the topological category of the
geometric realization (Corollary \ref{ABORCAT}).

\end{remark}

\begin{remark}

It is interesting to point out that Theorem \ref{SUBDIV} can be
extended to the geometrical simplicial category, that is,
$\gscat(\sd K)\leq \gscat K$ with an analogous argument. In
addition, notice that the complex $\sd K_5$ (Figure \ref{SDK5}) can
be covered with two trees and thus, $\gscat K_5=1$.
\end{remark}


\section{Geometric realization}\label{4}
\subsection{}
A natural question is to compare the simplicial category of the complex $K$ with the LS-category of  $X=\geo{K}$, the so-called geometric realization of $K$ \cite[\textsection 3.1]{SPANIER1966}.

We recall the classical definition of Lusternik-Schnirelmann
category of a topological space. For general properties of this
invariant we refer to \cite{CLOT2003}.

\begin{definition}\label{CATOP}
An open subset $U$ of the topological space $X$ is called {\em
categorical} if $U$ can be contracted to a point inside the
ambient space $X$. The {\em LS-category} of $X$, denoted by $\cat
X$, is  the least integer $n\geq 0$ such that there is a covering
of $X$ by $n+1$ categorical open subsets.
\end{definition}

It is known (see Proposition 1.10 of \cite{CLOT2003}) that, when
the space $X$ is a normal ANR,  the categorical sets in
Definition \ref{CATOP} can be taken to be {\em closed} instead of
open. In particular, this is the case for the geometric
realization $X=\geo{K}$ of a finite abstract simplicial complex
(see \cite[\textsection II.4]{LW1969}, also \cite[p.~84]{DOLD1972}
and the references and comments in  \cite[p.~247]{JAMES1999}).



On the other hand, O. Randal-Wallis pointed out to the second
author the following example, showing that categorical  sets
can be rather pathological.

\begin{example}Let $X=[0,1]$ be the unit interval and let $F$ be the Cantor set. Then $F$ is contractible in $X$
and has not the homotopy type of a finite CW-complex. This is
because $F$ is totally disconnected but non-discrete (see
\cite[\textsection 5.1]{FP1990}).
\end{example}
Anyway, the following theorem shows that, when  $X=\geo{K}$ is the
geometric realization of an abstract simplicial complex $K$, the
LS category of $X$ can be computed by means of a closed
categorical covering whose sets are subcomplexes of $K$ in a
certain subdivision. This result is essentially stated, but
without proof, in Fox's paper \cite[\textsection 3]{FOX1941}.

\begin{theorem}\label{Realization}If $X=\geo{K}$ then $\cat X\leq n$
if and only if there exist  subcomplexes $L_0,\dots,L_n$ of some
subdivision $K^\prime$ of $K$, such that each $\geo{L_j}$ is
contractible in $X$ and $X=\geo{L_0}\cup\cdots\cup\geo{L_n}$.
\end{theorem}

The proof of the above Theorem \ref{Realization} can be sketched
as follows: If we take $\delta$, the Lebesgue number of the
considered categorical covering and  we consider a generalized
subdivision of $K$, $\sd K$, such that $\mesh(\sd^n K)< \delta$.
Then, by considering the subcomplexes obtained with simplices
contained on each open element of this covering, a categorical
covering by subcomplexes is obtained.

\subsection{} We now state the precise relation between the simplicial category of a complex and the topological LS-category of its geometric realization.

\begin{theorem} Let $K$ be a finite simplicial complex. Then $\cat \geo{K}\leq \scat K$.
\end{theorem}

\begin{proof} Let $\scat K=n$ and let $\{U_0,\dots,U_n\}$ be a categorical simplicial covering of $K$, that is, each inclusion $I_j\colon U_j \to K$ is in the same contiguity class that a constant map $v_j$. Therefore, the induced maps $\geo{I_j}, \geo{v_j}\colon \geo{U_j}\to \geo{K}$
between the geometric realizations are homotopic, $\geo{I_j}\simeq\geo{v_j}$ \cite[\textsection 3.4]{SPANIER1966}. It is clear that $\geo{I_j}$ is the inclusion $\geo{U_j}\subset \geo{K}$, and $\geo{v_j}$ is a constant map. Thefore the subspaces $\geo{U_0},\dots,\geo{U_n}$ form a categorical closed cover of $\geo{K}$. As we commented before, since $\geo{K}$ is a  normal ANR we can consider closed covers instead of open covers. Then $\cat \geo{K}\leq n$.
\end{proof}

\begin{corollary}\label{sd} $\cat \geo{K}\leq \scat(\sd^NK)$, for any iterated barycentric subdivision $\sd^NK$ of $K$.
\end{corollary}

\begin{proof}
The geometric realizations $\geo{\sd K}$ and $\geo{K}$ are homeomorphic (see for instance \cite[Prop. 2.33]{KOZLOV2008}).
\end{proof}

\begin{example} The inequality in Corollary \ref{sd} may be strict. For instance, the
complex $K$ in Figure \ref{MOTHER} is topologically contractible
(it is collapsible), that is $\cat \geo{K} = 0$, but all the
barycentric subdivisions have simplicial category one. This
happens because, as proven by Barmak and Minian  in \cite[Theorem
4.15]{BARMAKMINIAN2012}, a complex is strongly collapsible if and
only if its barycentric subdivision is strongly collapsible. In
other words, $\scat K = 0$ if and only if $\scat (\sd K) = 0$. But
the complex $K$ is not strongly collapsible (it has no dominated
vertices) and can be covered by two categorical subcomplexes, so
$\scat K = 1$. By applying Theorem \ref{SUBDIV} we have $0 < \scat
(\sd^N K) \leq 1$ for all $N$.
\end{example}


\section{Products}\label{5}

\subsection{}\label{PRELIMPROD}  We recall the definition of categorical product (see Remark \ref{CATVSCART} below for an explanation of our notations).

\begin{definition}\label{PROD}
Let $K,L$ be two abstract simplicial complexes. The {\em categorical
product} $K\times L$  is defined as follows. The vertices of
$K\times L$ are the pairs $(v,w)$ of vertices with $v\in K$ and
$w\in K$. The simplices of $K\times L$ are the sets of vertices
$\{(v_1,w_1),\dots,(v_q,w_q)\}$ such that $\{v_1,\dots,v_q\}$ is a
simplex of $K$ and $\{w_1,\dots,w_q\}$ is a simplex of $L$.
\end{definition}

We fix some notation.

For $n\geq 1$, we shall denote
$K^n=K\times\stackrel{n)}{\cdots}\times K$. By definition, the
projections $\pi_j\colon K \to K^n$, $j=1,\dots,n$, onto each
factor are simplicial maps. The diagonal map $\Delta \colon K \to
K^n$ defined by $\Delta(v)=(v,\dots,v)$ is   simplicial  too. A
map $L\to K^n$ is simplicial if and only if the compositions with
all the projections are simplicial maps.

If $\varphi_1,\dots,\varphi_n\colon K \to L$ are simplicial maps,
the map
$$(\varphi_1\dots,\varphi_n)\colon K \to L^n$$ defined by
$$(\varphi_1\dots,\varphi_n)(v)=(\varphi_1(v)\dots,\varphi_n(v))$$ is simplicial.

If $\varphi\colon K\to L$ is a simplicial map then we denote by
$\varphi^n\colon K^n \to L^n$ the map
$$\varphi^n(v_1,\dots,v_n)=(\varphi(v_1),\dots,\varphi(v_n)),$$ which is simplicial.

\begin{remark}\label{CATVSCART}
Our Definition \ref{PROD} is like Definition 4.25 in
\cite{KOZLOV2008} but we  changed the notation $K\prod L$ to
$K\times L$ for the sake of simplicity.   As an example,
$\Delta^1\times\Delta^1\cong \Delta^3$. Note that $\geo{K^n}$ is
not homeomorphic to $\geo{K}^n$. Hence the categorical (or
simplicial) product should not  be confused with the so-called
cartesian product \cite[\textsection II.8]{ES1952}, which depends
on some ordering of the vertices and does not verify the universal
property of a product. However, as Kozlov proves in \cite[Prop.
15.23]{KOZLOV2008}, the geometric realizations of both products
have the same homotopy type, then $\geo{K^n}\simeq\geo{K}^n$ as
topological spaces (see also \cite[Theorem 10.21]{KOZLOV2008}).
\end{remark}

\begin{proposition}\label{STUFF}\
\begin{enumerate}
\item
Let $\varphi,\psi\colon K \to L$ and
$\varphi^\prime,\psi^\prime\colon K^\prime \to L^\prime$ be
simplicial maps such that $\varphi\sim_c\psi$ and
$\varphi^\prime\sim_c\psi^\prime$. Then
$$\varphi\times\varphi^\prime \sim_c\psi\times\psi^\prime\colon K\times K^\prime\to L\times L^\prime.$$
\item
Let $\varphi\sim_c \psi\colon K \to L$ and
$\varphi^\prime\sim_c\psi^\prime \colon K \to L^\prime$, then
$$(\varphi,\varphi^\prime)\sim_c(\psi,\psi^\prime)\colon K \to L\times L^\prime.$$
\end{enumerate}
\end{proposition}

\begin{corollary}\label{CONTPROD} Let $K\sim L$ be two complexes with the same strong homotopy type. Then $K^n\sim L^n$.
\end{corollary}

%
%
%

The following Theorem establishes the simplicial analogue of a
well known result on the LS-category of a product of topological
spaces.

\begin{theorem}
Let $K$ and $L$ be finite simplicial complexes. Then
$$\scat(K\times L)+1\leq (\scat K+1)(\scat L+1).$$
\end{theorem}
\begin{proof}
Suppose that $\scat K=n$ and $\scat L=m$, therefore there exists a
categorical covering   $U_0,...,U_n$ of $K$ and a categorical
covering $V_0,...,V_m$ of $L$. Consider the subcomplexes
$U_i\times V_j\subset  K\times L$, for $0\leq i \leq n$ and $0\leq
j\leq m$. We want to show that $U_i\times V_j$ form a  categorical
covering of $K\times L$.

Each inclusion map $i_{U_i}$ is in the same contiguity class of a
constant map $c_{u_i}$ where $u_i$ is a vertex of $K$;
analogously,  each inclusion $i_{V_j}$ is in the same contiguity
class of a constant map $c_{v_j}$ where $v_j$ is a vertex in $L$.
By  Proposition \ref{STUFF} the map $i_{U_i}\times i_{V_j}\colon
U_i\times V_j\rightarrow K\times L$ is in the same contiguity
class of the map $c_{v_i}\times c_{w_j}\colon U_i\times
V_j\rightarrow K\times L$. Clearly, $i_{U_i}\times
i_{W_j}=i_{U_i\times V_j}$ and  $c_{v_i}\times
c_{w_j}=c_{(v_i,w_j)}$, where  $(v_i,w_j)$ is a vertex of $K\times
L$. Therefore the subcomplexes $U_i\times V_j$ are categorical.

Now, we shall prove that $\{U_i\times V_j\}$ is a covering of
$K\times L$. If the simplex $\{(v_{0},w_{0}),...,(v_q,w_q)\}$ is
in $K\times L$ then $\{v_{0},...,v_q\}$ is contained in a
subcomplex $U_i$ of $K$ and $\{w_{0},...,w_q\}$ is contained in a
subcomplex $W_j$ of $L$. Then $\{(v_{0},w_{0}),...,(v_q,w_q)\}$ is
contained in $U_i\times V_j$. Thus, we conclude that
$\scat(K\times L)+1\leq (n+1)(m+1)$.
\end{proof}

\section{Whitehead construction}\label{6}
\subsection{}  It is well known that for topological spaces $X$ with ``good properties'' there is the following so-called Whitehead characterization of the topological LS-category (see \cite[Theorem~1.55]{CLOT2003}):
\begin{theorem}
$\cat X\leq n$ if and only if the diagonal map $\Delta \colon X
\to X^{n+1}$ factors (up to homotopy) through the so-called {\em
fat wedge} $T^{n+1}X$.
\end{theorem}

This result is a very useful tool for computing LS-category. In
this section we shall try to adapt it to abstract simplicial
complexes.

First we define a simplicial version of the topological fat wedge
$T^nX$ of a topological space $X$ \cite[\textsection
1.6]{CLOT2003}. Also, we shall briefly develop the notion of {\em
pointed} contiguity class.

Let $K$ be an abstract simplicial complex and fix some vertex
$v_0$ of $K$ as a base point. For each $j=1,\dots,n$ let
\begin{equation}\label{SUBCOM}
K_j=\pi_j^{-1}(\{v_0\})= K\times\cdots \times
\{v_0\}\times\cdots\times K,
\end{equation}
be the subcomplex  of $K^n$ spanned by the vertices whose $j$-th
coordinate is the base point $v_0$.

\begin{definition}For $n\geq 1$ the {\em $n$-th fat wedge} $T^nK$ is the  subcomplex $K_1\cup\cdots\cup K_n\subset K^n$.
\end{definition}

For instance, $T^1K=\{v_0\}$ is a point and $T^2K$ is the wedge $K
\vee K$. Note that $T^nK$ is not a full subcomplex of $K^n$.

\medskip

A pointed map (that is, a simplicial map preserving the base
points) $\varphi\colon (K,v_0)\to (L,w_0)$ induces a simplicial
map $T^n\varphi\colon T^nK \to T^nL$, which is the restriction of
$\varphi^n\colon K^n \to K^n$.

\begin{proposition}Let $\varphi,\psi\colon K \to L$ be two contiguous simplicial maps preserving the base points. Then the induced maps $T^n\varphi,T^n\psi\colon T^nK\to T^nL$ are contiguous.
\end{proposition}

\begin{proof}  The maps $\varphi^n,\psi^n\colon K^n \to L^n$ are contiguous by Proposition \ref{STUFF}. Moreover they send each subcomplex  $K_j$  into itself.\end{proof}

\begin{corollary}\label{INVPOINT}Let $(K,v_0) \sim (L,w_0)$ be two pointed simplicial complexes with the same pointed strong homotopy type. That is, we assume that the
homotopy equivalences $\varphi,\psi$ between $K$ and $L$, as well
as the sequences of contiguous maps defining the relations
$\psi\circ\varphi\sim 1_K$ and $\varphi\circ\psi\sim 1_L$,
preserve the base points. Then $T^nK\sim T^nL$.
\end{corollary}


\subsection{} We are now in position to discuss the Whitehead formulation of the simplicial LS-category. In order to be systematic
we follow the approach of \cite[\textsection 1.6]{CLOT2003}, by
defining a so-called {\em simplicial Whitehead category} $\wscat
K$ and trying to compare it to the simplicial LS-category $\scat
K$.

\begin{definition}\label{WHITE}We say that $\wscat K\leq n$
if the diagonal map $\Delta\colon K \to K^{n+1}$ factors through
the fat wedge $T^{n+1}K$ up to contiguity class. That is, there
exists some simplicial map $\delta\colon K \to T^{n+1}K$ such that
$I\circ \delta \sim \Delta$, where we denote by $I\colon T^{n+1}K
\subset K^{n+1}$ the inclusion (see Figure \ref{WHITEHEAD}).
\end{definition}

\begin{figure}[h]
$\xymatrix{
\ T^{n+1}K\ \ar@{^{(}->}[r]^{\ I\ } &\ K^{n+1}\\
K\ar@{-->}[u]^{\delta}\ar[ru]_{\Delta}&}$ \caption{Simplicial
Whitehead LS-category} \label{WHITEHEAD}
\end{figure}

\begin{theorem}\label{INEQ} $\scat K\leq \wscat K$.
\end{theorem}
\begin{proof}
Assume that $\wscat K=n$, and let $\delta\colon K \to T^{n+1}K$ be
as in Definition~\ref{WHITE}. Let $\pi_j\colon K^{n+1} \to K$ be
the projections onto each factor. Let $K_j\subset T^{n+1}K$,
$j=1,\dots,n+1$,  be the subcomplexes defined in (\ref{SUBCOM}).

Since $I\circ \delta \sim \Delta$, there exists a sequence $I\circ
\delta=\varphi_1, \varphi_2,\dots,\varphi_m=\Delta$ of maps such
that $\varphi_i$ and $\varphi_{i+1}$ are contiguous. Call
$L_j=\Delta^{-1}(K_j)$  the preimages of the $K_j\subset
T^{n+1}K$ defined above. Clearly, $K=L_1\cup\cdots\cup L_{n+1}$.
It only remains to show that each subcomplex $L_j\subset K$ is
categorical, that is, each inclusion map $I_j\colon L_j\subset K$
is in the same contiguity class that a constant map. Let us prove
it:

Since $\varphi_i\sim_c \varphi_{i+1}\colon K\to K^{n+1}$ it
follows that
$$\pi_j\circ\varphi_i\circ I_j\sim_c \pi_j\circ\varphi_{i+1}\circ I_j\colon L_j\to K.$$
Now, $\pi_j\circ\varphi_1\circ I_j=\pi_j\circ I\circ \delta\circ
I_j$, which is the constant map $c_{v_0}$ because
$\Delta(L_j)\subset K_j$. On the other hand,
$\pi_j\circ\varphi_m\circ I_j=\pi_j\circ\Delta\circ I_j$ is the
inclusion $I_j$. Then $I_j\sim c_{v_0}$.

We have found $n+1$ categorical subcomplexes covering $K$,
therefore $\scat K\leq n$.
\end{proof}


\subsection{} In this section we shall prove that the converse inequality of Theorem \ref{INEQ}
is not true, by exhibiting an example of a complex $K$ such that
$\scat K=1$ while $\wscat K>2$. The proof of the next propositions
is inspired by a result about finite co-H-spaces proved in
\cite{HV}.  We thank J.~Oprea for pointing out this reference to
us and R.D. Helmstutler for some explanations about his paper.

First we prove that the simplicial Whitehead category is an
invariant of the pointed strong homotopy type.

\begin{proposition}If $(K,v_0)\sim (L,v_0)$ is a pointed strong equivalence as in Corollary \ref{INVPOINT}, then $\wscat K=\wscat L$.
\end{proposition}

\begin{proof}Let $\wscat L=n$. Consider the diagram in Figure \ref{WHITEHEADKL} and the strong equivalences $K^{n+1}\sim L^{n+1}$ (Corollary \ref{CONTPROD}) and $T^{n+1}K\sim T^{n+1}L$ (Corollary \ref{INVPOINT}).

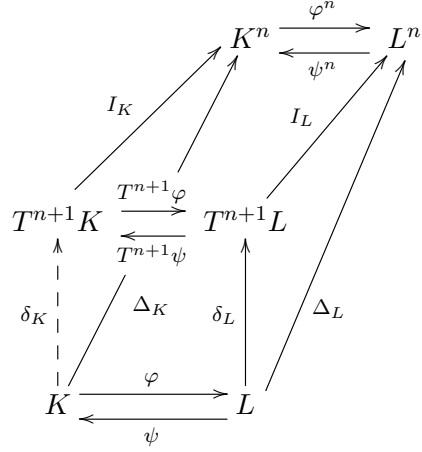
\begin{figure}[h]
\begin{center}
$\xymatrix{
&\ K^n\ar@<+1ex>[r]^{\varphi^n}&\ L^n\ar@<+1ex>[l]^{\psi^n}&&\\
&&&\\
{\ T^{n+1}K\ }\ar@<+1ex>[ruu]^{\ I_K}\ar@<+1ex>[r]^{T^{n+1}\varphi}&{\ T^{n+1}L\ }\ar@<+1ex>[l]^{T^{n+1}\psi} \ar[ruu]^{\ I_L}&&\\
&&&\\
K\ar@{-->}[uu]^{\delta_K}\ar[ruuuu]_(0.3){\
\Delta_K}|{\hueco}\ar@<+1ex>[r]^{\varphi}&L\ar@<+1ex>[l]^\psi\ar[uu]^{\delta_L}\ar@<-1ex>[ruuuu]_(0.3){\Delta_L\
}&& }$ \caption{Homotopy invariance} \label{WHITEHEADKL}
\end{center}
\end{figure}

With the obvious notations, from the relations
$$\psi\circ\varphi\sim 1_K, \quad \varphi\circ\psi\sim 1_L,$$ it follows that
$$\psi^{n+1}\circ\varphi^{n+1}\sim 1_{K^{n+1}}, \quad \varphi^{n+1}\circ\psi^{n+1}\sim 1_{L^{n+1}},$$ as well as
$$T^{n+1}\psi\circ T^{n+1}\varphi\sim 1_{T^{n+1}K}, \quad T^{n+1}\varphi\circ T^{n+1}\psi\sim 1_{T^{n+1}L}.$$

Moreover, from the definitions, we have
$$I_L\circ T^{n+1}\varphi=\varphi^n\circ I_K, \quad I_K\circ T^{n+1}\psi=\psi^n\circ I_L,$$ as well as
$$\Delta_L\circ\varphi= \varphi^n\circ\Delta_K, \quad \Delta_K\circ\psi=\psi^n\circ\Delta_L.$$

Define
$$\delta_K\eqdef T^{n+1}\psi\circ\delta_L\circ\varphi.$$

Since, from hypothesis, $I_L\circ\delta_L\sim\Delta_L$, it follows
\begin{align*}
I_K\circ\delta_K&=(I_K\circ T^{n+1}\psi)\circ\delta_L\circ\varphi\\
&=\psi^n\circ (I_L\circ\delta_L)\circ\varphi\sim(\psi^n\circ \Delta_L)\circ\varphi\\
&=\Delta_K\circ(\psi\circ\varphi)\sim 1_K.
\end{align*}
Then $\wscat K\leq \scat L$. The other inequality is proved in the
same way.
\end{proof}

Clearly, we may assume that the base point $v_0$ is in the core
$K_0$, hence eliminating dominated vertices  is a pointed
equivalence $(K,v_0)\sim (K_0,v_0)$.

\begin{corollary}\label{CORE}The simplicial Whitehead category of a complex equals that of  its core,
$\wscat K=\wscat K_0$.
\end{corollary}

We are now in a position to prove the main results of this
section.

\begin{lemma}{\cite[Prop.~2.7]{BARMAKMINIAN2012}}\label{MINID} Let $K_0$ be a minimal complex and let $\varphi\colon K_0 \to K_0$ be a simplicial map which lies in the same contiguity class as the identity.
 Then $\varphi$ is the identity.
\end{lemma}

\begin{theorem}\label{ONEZERO}Let $K$ be a simplicial complex such that $\wscat K\leq 1$. Then $K$ is strongly collapsible, which is equivalent to $\scat K=0$.
\end{theorem}

\begin{proof}If $\wscat K=0$ then the result follows from Theorem \ref{INEQ}.

If $\wscat K=1$ we have, from Corollary \ref{CORE}, that $\wscat
K_0=1$. This means that there exists a simplicial map
$\varphi\colon K_0 \to T^2K_0$ such that $i\circ \varphi\sim
\Delta$, where $\Delta\colon K_0\to (K_0)^2$ is the diagonal map.
Let $\pi_1,\pi_2 \colon (K_0)^2 \to K_0$ be the projections of the
categorical product. Then
$$i\circ\varphi \sim \Delta \implies \pi_1\circ i\circ\varphi \sim \pi_1\circ \Delta =1\implies  \pi_1\circ i\circ\varphi  =1$$
because $K_0$ is a minimal complex (Lemma \ref{MINID}). Remember
that in the proof of Theorem \ref{INEQ} we denoted by $L_1$
(respectively, $L_2$) the subcomplex $\varphi^{-1}(\{v_0\}\times
K)$ (resp. $\varphi^{-1}(K\times \{v_0\})$).Then, for $v\in L_1$
we have $$v_0=\pi_1\circ i\circ \varphi(v)=v$$ because
$\varphi(v)\in L_1$, which shows that $L_1=\{v_0\}$. Analogously
$$\pi_2\circ i\circ \varphi =1=\pi_2\circ \Delta$$ proves that $L_2=\{v_0\}$. Then the core $K_0=L_1\cup L_2=\{v_0\}$ is a point, that is, the complex $K$ has the strong homotopy type of a point.
\end{proof}

\begin{example}The complex $K$ in \figref{MOTHER} has $\scat K=1$ so it is not strongly equivalent to a point. From Theorem \ref{ONEZERO} its  Whitehead simplicial category is at least $2$. Then $\wscat K>\scat K$.
\end{example}


\section{Cofibrations}\label{7}

\subsection{}We now briefly discuss the notion of cofibration in the simplicial setting.

The ``homotopy extension property'' is a very important notion in
topology. A cofibration is a map $A\to X$ which satisfies the
homotopy extension property with respect to all spaces.  It is
well known (see \cite[p. 14]{HATCHER2002} or \cite[p. 68]{LW1969})
that if $A$ is a subcomplex of a CW-complex $X$ then the pair
$(X,A)$ has the homotopy extension property. Therefore, if
$L\subset K$ is a subcomplex of a simplicial complex,  the pair
$(\geo{K},\geo{L})$ of the geometric realizations has the (topological) homotopy extension
property. We want to define a purely combinatorial analogue for a
simplicial pair $(K,L)$.

\begin{definition}
A simplicial map $i\colon L\to K$ has the {\em contiguity
extension property} if given two simplicial maps $\varphi,
\psi\colon L \to M$ which lie in the same contiguity class,
$\varphi\sim\psi$, and given an extension of $\varphi$ (that is, a
simplicial map  $\tilde\varphi \colon K \to M$ such that
$\varphi\circ i=\varphi$), there exists an extension $\tilde\psi$
of $\psi$ with $\tilde\varphi\sim\tilde\psi$ (cf. Figure
\ref{FIG}).

\begin{figure}[h]
$\xymatrix{
\ L\ \ar[dd]_i\ar@<+1ex>[rr]^{\ \varphi\ } \ar@<-1ex>[rr]_{\ \psi\ }&&\ M\\
&&\\
K\ar@<-1ex>[rruu]^{\
\tilde\varphi}\ar@{-->}@<-3ex>[rruu]_{\tilde\psi\ }&&}$
\caption{Contiguity extension property.} \label{FIG}
\end{figure}
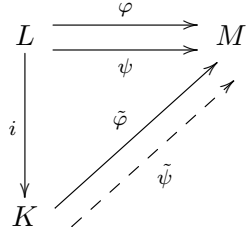

\end{definition}

\begin{example}Let $K_0$ be the core of the complex $K$. Then the inclusion $i\colon K_0\subset K$
is a simplicial cofibration. In fact, as explained before, there
is a simplicial retraction $r\colon K \to K_0$ such that $ r\circ
i =1_{K_0}$ and $i\circ r\sim 1_K$. Then one can take $\tilde \psi
=r\circ \psi$ because $\tilde \varphi\sim \varphi\circ r$.
\end{example}

As an application of Theorem \ref{ONEZERO}, we shall show that
there are simplicial pairs $(K,L)$ that do not have the extension
property.

\begin{theorem}Let $K$ be a connected simplicial complex with $\scat K=n$. Assume that there is a categorical covering $L_1,\dots L_{n+1}$ of $K$ such that all the pairs $(K,L_j)$
have the contiguity extension property. Then $\wscat K=n$.
\end{theorem}
\begin{proof}Due to Theorem \ref{INEQ} we only have  to prove that $\wscat K \leq n$.
By hypothesis, each inclusion $I_j\colon L_j \to K$ is in the same contiguity class that some constant map, whose image is a vertex $v_j$. Since two such constant maps are contiguous if and only if the vertices $v_i,v_j$ lie on the same simplex (Lemma \ref{TWOPOINTS}), and $K$ is connected, we can suppose that all the vertices are equal, say to some base point $v_0$. Taking $i=I_j$, $\varphi=I_j$, $\psi=v_0$ and $\tilde \varphi=1_K$, the simplicial extension property gives maps $\tilde\psi_j\sim 1_K$ such that $\tilde\psi_j(L_j)=\{v_0\}$. Define the map
$$\delta\colon K \to T^{n+1}K, \quad \delta(v)=(\tilde\psi_1(v),\dots,\tilde\psi_{n+1}(v)).$$
It is well defined, because  each vertex $v$ is contained in some $L_j$, hence $\tilde\psi_j(v)=v_0$, meaning that the $j$-th coordinate of $\delta(v)$ is the base point. Moreover $\delta$ is a simplicial map. In fact,  for any simplex $\sigma\in K$ we have that $\delta(\sigma)$ is a simplex in $K^n$ (Proposition \ref{PRELIMPROD}). Moreover, $\sigma$ must be contained in some $L_j$, and $\delta(L_j)\subset K_j$. Then $\delta(\sigma)\in K_j\subset T^{n+1}K$.
Now, let
$$\tilde\psi_j =\varphi_j^1\sim_c\cdots \sim_c\varphi_j^m=1_K$$
be the sequence of contiguous maps connecting $\tilde\psi_j$ and
$1_K$ (clearly we may assume that the length $m$ does not depend
on $j$). Then  the maps
$$(\tilde\varphi^k_1,\dots,\tilde\varphi^k_{n+1})\colon K \to
K^{n+1}, \quad k=1,\dots,m,$$ define a sequence of contiguous maps
between $J\circ \delta=(\tilde\psi_1,\dots,\tilde\psi_{n+1})$ and
$\Delta=(1,\dots,1)$. Hence $J\circ \delta\sim \Delta$, which
means that $\wscat K\leq n$.
\end{proof}

As a corollary, some categorical covering of the complex $K$ in Figure \ref{MOTHER} fails to verify the hypothesis of the preceding theorem, because $\scat K=1$ but $\wscat K>2$.

\begin{example}

The subcomplex $L=K\setminus\sigma$,  of the complex $K$ in Figure
\ref{MOTHER}, where  $\sigma=\{a,b,c\}$, has not the contiguity
extension property. Let $v$ be the upper vertex of $K$ (see figure
\ref{MOTHER}). Since $L$ is strongly collapsible to $v$,   the
inclusion $\varphi=i\colon L \to K$ and the constant map
$\psi=c_v\colon L \to K$ lie in the same contiguity class. Fix
$\tilde\varphi=1_K\colon K \to K$ to be the identity. If there
exists $\tilde\psi\sim 1_K$ such that $\psi(L)=\{v\}$ there must
be some sequence $\tilde\psi\sim_c \varphi_1\sim_c\dots\sim_c 1_K$
of contiguous maps. However, since $a,b,c\in L$, the simplicial
map $\varphi_1$ must be constant, and we obtain that $1_K$ is
contiguous to a constant map, which is a contradiction because
$\scat K\neq 0$.

\begin{figure}[h!]
\begin{center}
\includegraphics[height=40mm]{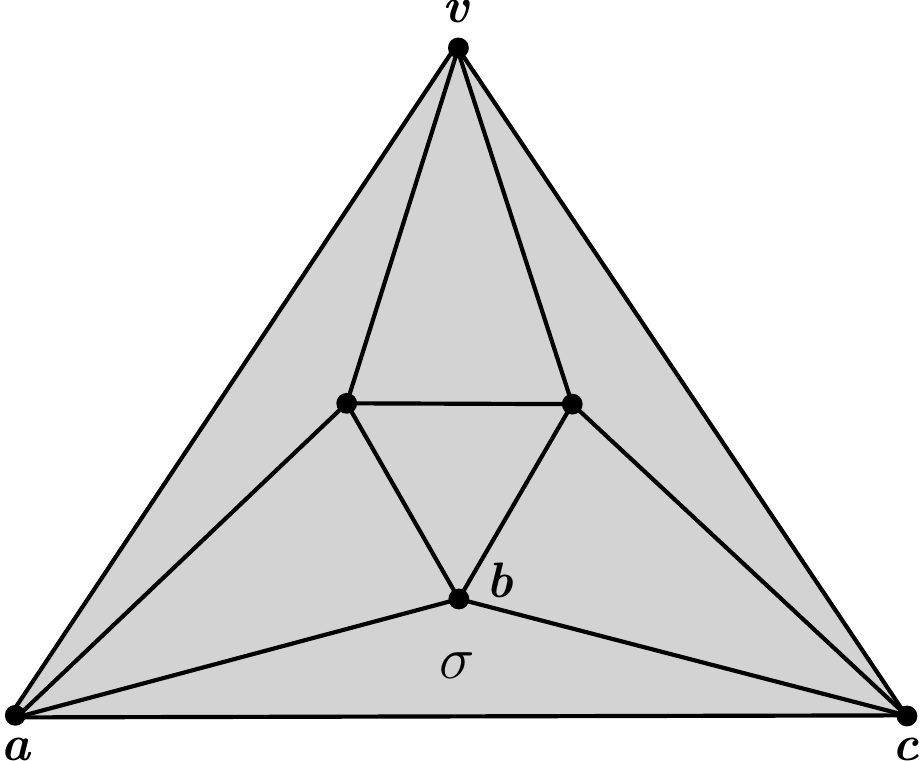}
\caption{A collapsible but not strongly collapsible complex $K$.
The subcomplex $L=K\setminus \sigma$ has not the contiguity
extension property. } \label{MOTHER}
\end{center}
\end{figure}

\end{example}

\begin{example}A simpler example was communicated to us by N. Scoville:
let $K$ be the $1$-dimensional complex in Figure \ref{TRIANG},
with vertices $a, b, c$ and let $\sigma$ be the $1$-simplex joining
the vertices $b$ and $c$. Then the subcomplex $L=K\setminus
\sigma$ has not the contiguity extension property.

\begin{figure}[h!]
\begin{center}
\includegraphics[height=25mm]{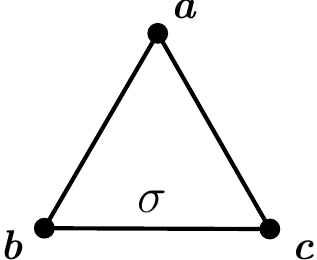}
\caption{The subcomplex $L=K\setminus \sigma$ has not the
contiguity extension property. } \label{TRIANG}
\end{center}
\end{figure}

\end{example}


\section{Graphs}\label{8}
\subsection{}
This section is focused on the study of the simplicial
LS-category in the one-dimensional case,
that is, on graphs. The well known graph-theoretical notion of
{\em arboricity} will play a central role in this study.
Basically, arboricity is based on the cardinality of minimal
decompositions of a graph into disjoint spanning forests, i.e.,
acyclic subgraphs, which are non-necessarily connected and cover
all the vertices.

The aim of this section is to prove that arboricity coincides (up to one) with
both simplicial and geometric simplicial categories.

\begin{remark}
In \cite{AS}, Aaronson and Scoville introduced a so-called discrete LS-category in the simplicial seting. They proved that, for the $1$-dimensional case, it is equivalent to arboricity.
\end{remark}

Let us start with some basic notions on graph theory. A general
reference is Harary's book \cite{HARARY1969}.

\begin{definition}
Let $G$ be a graph. A cycle in $G$ is an alternating sequence of
distinct vertices and edges, $v_0,e_1,v_1,\dots ,
v_{n-1},e_n,v_n$, where the incident vertices of each edge $e_i$
are $v_{i-1}$ and $v_i$ respectively, and such that $v_0=v_n$.
\end{definition}

Under a topological point of view, cycles are triangulations of
the circumference $S^1$.

\begin{definition}
A {\em forest}  is a graph without cycles, alternatively it can be
called acyclic graph. A {\em tree} is a connected forest.
\end{definition}

\begin{definition}
The {\em arboricity} of a graph $G$, denoted by $\Upsilon(G)$, is
the minimum number of edge-disjoint spanning forests into which
$G$ can be decomposed.
\end{definition}

Nash-Williams \cite{NW} determined the arboricity of a general graph:

\begin{theorem}{\cite[Th. 9.10]{HARARY1969}}
Let $G$ be a nontrivial graph and let $q_n$ be the maximum number
of edges in any subgraph of $G$ with $n$ vertices. Then
$$ \Upsilon(G)=\max_n\left \lceil  \frac{q_n}{n-1}  \right \rceil .$$
\end{theorem}

\begin{example} For the particular case of complete graphs (see Figure \ref{K5}) it follows the following formula:
$$\Upsilon(K_{2n})=n=\Upsilon(K_{2n-1}).$$
\end{example}


\subsection{}Now we prove some   results which will be used later
in this section.

\begin{remark}
Every standard elementary collapse in a graph is the deletion of a
so-called {\em leaf vertex} $v$ and the unique edge $vv'$
containing $v$. Thus $v$ is dominated by $v'$ and hence, in
graphs, every elementary collapse is an elementary strong
collapse.

\end{remark}

\begin{lemma}\label{CYCLE}
Let $G$ be a connected graph and let $L\subset G$ be a subgraph
containing at least one cycle $C$. If $\varphi\colon L \to G$ is a
simplicial map contiguous to the inclusion $i_L\colon L\to G$,
then $\varphi(L)$ contains the cycle $C$. Moreover
$\varphi(L)\subset L$ and $\varphi_{\vert C}$ is the inclusion
$i_C\colon C \to G$.
\end{lemma}

\begin{proof}
Every edge $e$ in $L$ satisfies that $\varphi(e)\cup e$ is a
simplex in $G$ (for graphs, ``simplex'' means a vertex or an
edge), so, equivalently $\varphi(e)\subseteq e$; therefore, either
$\varphi(e)=e$ or $\varphi(e)$ is one of the extreme vertices of
$e$. This implies that $\varphi(L)\subseteq L$.

Let $C$ be a cycle contained in $L$ and let us consider the
restrictions to $C$ of $\varphi$ and the inclusion $i_L$, denoted
by $\varphi_{\vert _C}$ and $i_C$ respectively, which are also
contiguous as consequence of the fact that the composition of
contiguous maps is contiguous (see \cite[\textsection
3.5]{SPANIER1966}). So, if $\varphi$ maps every edge of the cycle
$C$ onto itself then $\varphi(C)=C$ and hence $\varphi(L)$
contains the cycle $C$. Otherwise,  there exists an edge $e_1$ in
$C$ such that $\varphi(e_1)\neq e_1$; we can suppose without loss
of generality that $\varphi(e_1)= v_0$. Now let us consider the
edge $e_2$ which is adjacent to $e_1$ in $v_1\neq v_0$. Since the
map $\varphi$ is simplicial, we have $\varphi(v_1)=v_0$, which is
a contradiction with $\varphi(e_2)\subseteq e_2$. Finally, we
conclude that all the edges in $C$ remain fixed by $\varphi$ so
$\varphi(L)$ contains at least one cycle.
\end{proof}

\begin{remark}
Let $P\subset G$ be a a path in a graph, such that there is a
simplicial map $\varphi\colon P \to G$ contiguous to the inclusion
$i_P\colon P\to G$. By the same argument of the proof above,
any edge $e$ contained in   $P$ satisfies that either
$\varphi(e)=e$ or $\varphi(e)=u$, where $u$ is one of the extreme
vertices of $e$ with degree $1$, i.e. $u$ is a so called leaf vertex of $P$.
Thus, we conclude that the only possible reductions induced by a
simplicial map contiguous to the inclusion are given by standard
collapses.
\end{remark}

The next result establishes the equivalence for graphs between the
notions of categorical subcomplex and acyclic subgraph.

\begin{theorem}\label{CATFOR}
Let $G$ be a connected graph and let $L\subseteq G$ be a subgraph.
Then $L$ is categorical in $G$ if and only if $L$ is a forest.
\end{theorem}

\begin{proof}
Let us suppose that there exists a  categorical but non-acyclic
subgraph $L\subset G$.  By definition, there exists a vertex $v\in
G$ such that the inclusion $i=i_L\colon L \to G$ and the constant
map $c=c_v\colon L \to G$ are in the same contiguity class, which
gives a sequence $i_L=\varphi_1\sim_c\cdots\sim_c\varphi_m=c_v$ of
directly contiguous maps $\varphi_i\colon L \to G$.

Now, since $i_L\sim_c \varphi_2$ and $L$ contains
 at least one cycle $C$, by taking into account the previous Lemma we conclude that $\varphi_2(L)$ contains at least the cycle $C$. Moreover, $(\varphi_2)_{\vert C}=i_C$ is the inclusion $i_C\colon C\subset G$. Now, $\varphi_2\sim_c \varphi_3$  implies, by composing with the inclusion $\varphi_2(L)\subset L$,  that $i_C=(\varphi_2)_{\vert C}\sim_c (\varphi_3)_{\vert C}$.  That means, by applying the lemma again, that $\varphi_3(C)$ is the cycle $C$, and by repeating the argument we shall arrive to a cycle that can be deformed into a point, which is impossible.
Hence we conclude that $L$ is an acyclic subgraph of $G$, that is,
$L$ is a forest.

Conversely, let us assume that $L$ is a forest, so $L$ is a
disjoint union of trees $T_i$ with $i=1,\dots,n$. It is clear
that each inclusion map $T_i\subset G$ is in the same contiguity
class as the constant map
 sending the tree $T_i$ onto one of its vertices $v_i$.
Since $G$ is connected, there is at least one path in $G$ joining
every vertex $v_i$ with a given vertex $v_0$. Then the inclusion
of $L$ is in the   same contiguity class as the constant map $v$.
Hence $L$ is a categorical subcomplex in $G$.
\end{proof}

\begin{lemma}\label{Arbol}
Let $G$ be a connected graph. For every covering of $G$ by
forests there exists a covering of $G$ by trees with
the same number of elements.
\end{lemma}

\begin{proof}
Let $F_1,\dots,F_k$ be a covering of $G$ by forests. Since $G$ is
connected, it follows that given two trees $T$ and $T'$ in a
forest $F_i$ and any vertices $v\in T$ and $v'\in T'$, they can be
linked by means of a path in $G$. By adding such a path to $F_i$,
and removing, if necessary, the edges of the path which create
cycles and are not contained in $F_i$, we link the trees $T$ and
$T'$ of $F_i$ (see Figure \ref{BOSQUE}).

\begin{figure}[h]
\begin{center}
\includegraphics[width=0.45\textwidth]{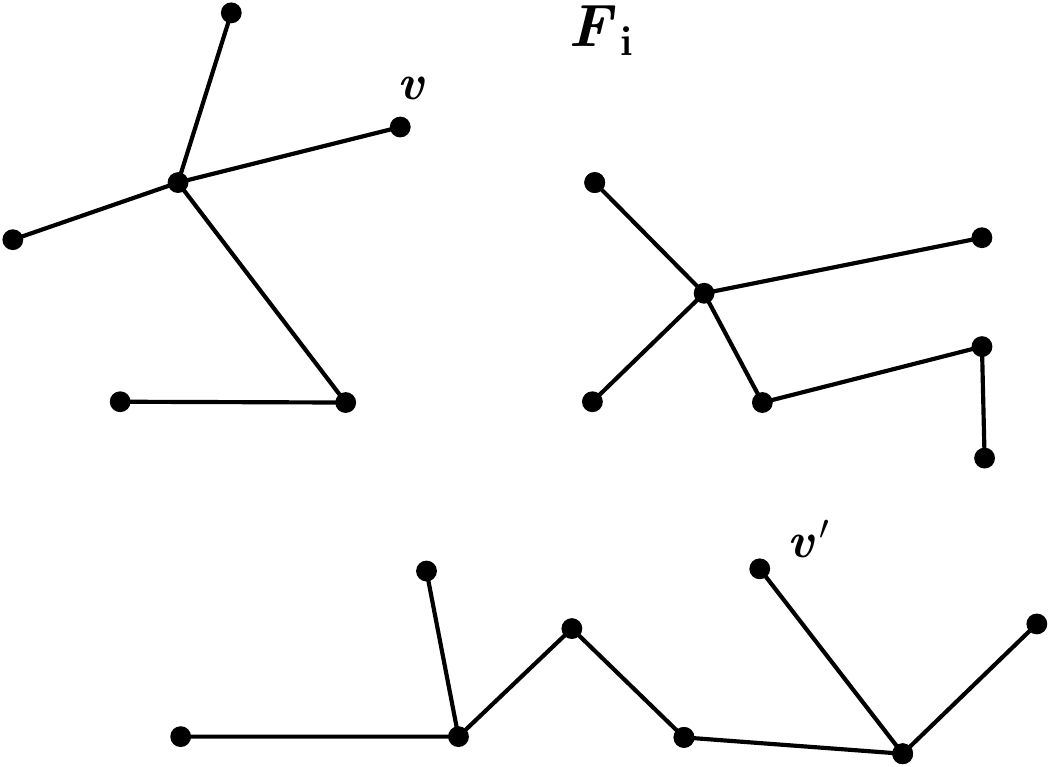}\hspace*{1cm}
\includegraphics[width=0.45\textwidth]{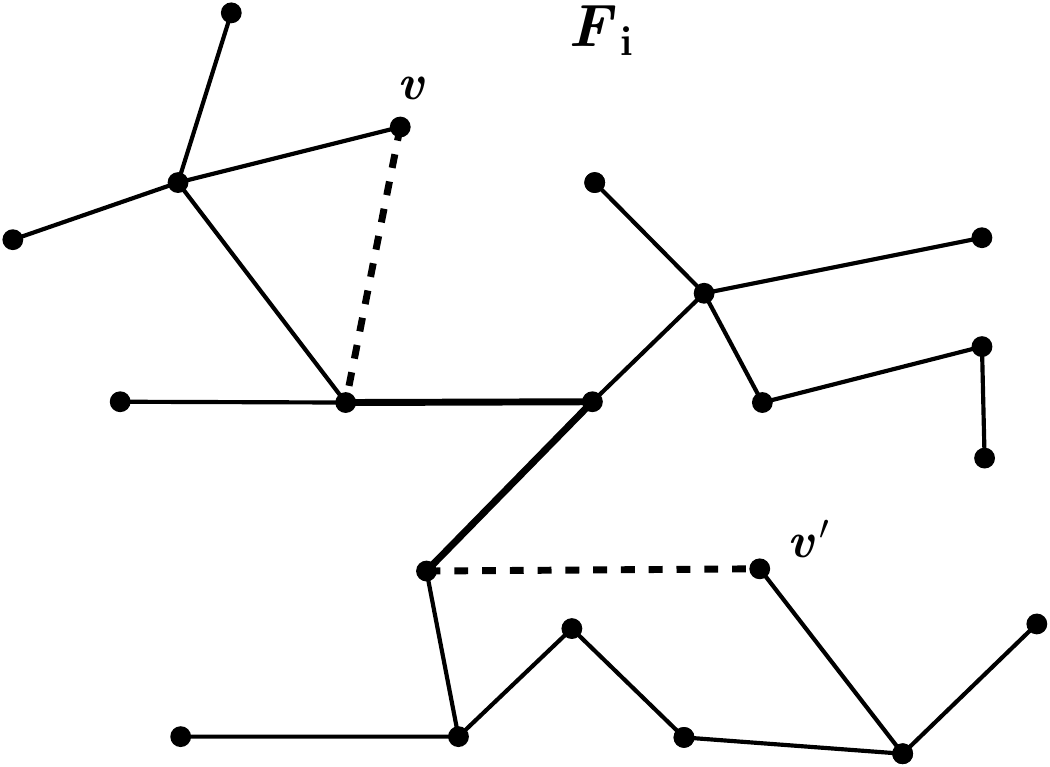}
\end{center}
\caption{Converting a forest into a tree.}\label{BOSQUE}
\end{figure}

Repeating the argument for the remaining trees, we obtain one tree
$T_i$ containing $F_i$. Finally, carrying out the same procedure
on every forest of $G$, we obtain a covering of $G$ by trees
$T_1,\dots,T_k$.
\end{proof}


\subsection{}We now state the main results of this section.

\begin{theorem}\label{ARBCAT}
Let $G$ be a connected graph. Then $\gscat G=\Upsilon(G)-1$.
\end{theorem}

\begin{proof}
Let us suppose that $\Upsilon(G)=k+1$, so by definition there
exists a covering of $G$ with $k+1$ edge-disjoint spanning
forests. By means of  Lemma \ref{Arbol} we can construct a
covering of $G$ with $k+1$ trees. Since the trees are strongly
collapsible   we conclude that $\gscat G\leq k$.

Conversely, let us assume that $\gscat G=k$. It means that there
is a covering of $G$ with $k+1$ strongly collapsible subsomplexes.
By Theorem \ref{CATFOR}, these complexes are trees $T_0,\dots,T_k$.
Starting from a tree $T_i$, we can obtain a spanning forest $F_i$
by adding all the isolated vertices which are not covered by
$T_i$. Next, if an edge is contained in several forests, in order
to obtain a covering by edge-disjoint forests, we remove it from
all these forests but one. Hence $\Upsilon(G)\leq k+1$.
\end{proof}

\begin{corollary}\label{ABORCAT}For any graph $G$, we have $\scat G=\gscat G=\Upsilon(G)-1$.
\end{corollary}

\begin{proof}
Apply again  Lemma \ref{CYCLE} and Theorem \ref{CATFOR}
separately on each connected component of $G$.
\end{proof}


\subsection{} In section 2.4, the behaviour of $\scat$ under barycentric
subdivisions was studied. Actually, it was proved that $\scat(\sd
K) \leq \scat K$. In the one-dimensional case there are examples
where this inequality is strict (see Example \ref{SUBDIVGRAPH}).
Also there are examples of $2$-dimensional complexes where $\sd^N
K$ remains constant for all $N$ and is strictly greater than the
topological LS-category of the geometric realization $\cat\geo{K}$
(see Example \ref{SUBDIVGRAPH}).

In contrast, in the one-dimensional case it holds the geometric
category of the first barycentric subdivision always reaches the
topological category of the geometric realization. In fact, the
following result states that this is possible for a certain kind
of ``local'' barycentric subdivision.

Remember that any finite CW complex $X$ satifies $\cat X\leq\dim
X$, therefore $\cat \geo{K}= 1$ for a non-contractible graph.

\begin{proposition}
Let $G$ be a connected graph. If $G$ is a tree then $\scat G =\cat
\geo{G}=0$. Otherwise,   let $G^\prime$ be the subdivision
obtained from $G$ by only bisecting those edges out of a spanning
tree in $G$. Then $\scat G^\prime=\cat\geo{G}=1$.
\end{proposition}

\begin{proof}
Let us consider a spanning tree $T$ in $G$. It is well known
\cite[Prop. 1A.2 ]{HATCHER2002} that there is a bijection between
the edges out of a spanning tree and the basic cycles generating
the one-dimensional homology of a graph. Now, after every edge out
of $T$ is barycentrically subdivided,  the subdivision $G^\prime$
of $G$ can be covered by two different (non disjoint) spanning
trees $T_1$ and $T_2$ constructed as follows (see Figure
\ref{SDK5}): $T_1$ is an expansion of $T$ which is obtained by
adding one edge (one of the subdivided ones) on every leaf vertex
of $T$; analogously, $T_2$ is obtained by adding the other edge
(not previously added to construct $T_1$) to  all the leaf
vertices of $T$. By definition both $T_1$ and $T_2$ are spanning
trees covering $G^\prime$ and hence $\scat G^\prime=\cat
\geo{G}=1$.
\end{proof}

\begin{example}
Notice that for certain graphs, as the complete ones $K_n$, the
simplicial LS-category equals the topological LS-category of the
realization of the graph by bisecting a fewer number of edges than
the stated one by the above Proposition. For example, for $K_5$
only three edges must be bisected in order to get $\scat K'_5=1$
(see Figure \ref{TRESEDGES}).

\begin{figure}[h]
\begin{center}
\includegraphics[
height=4.5cm]{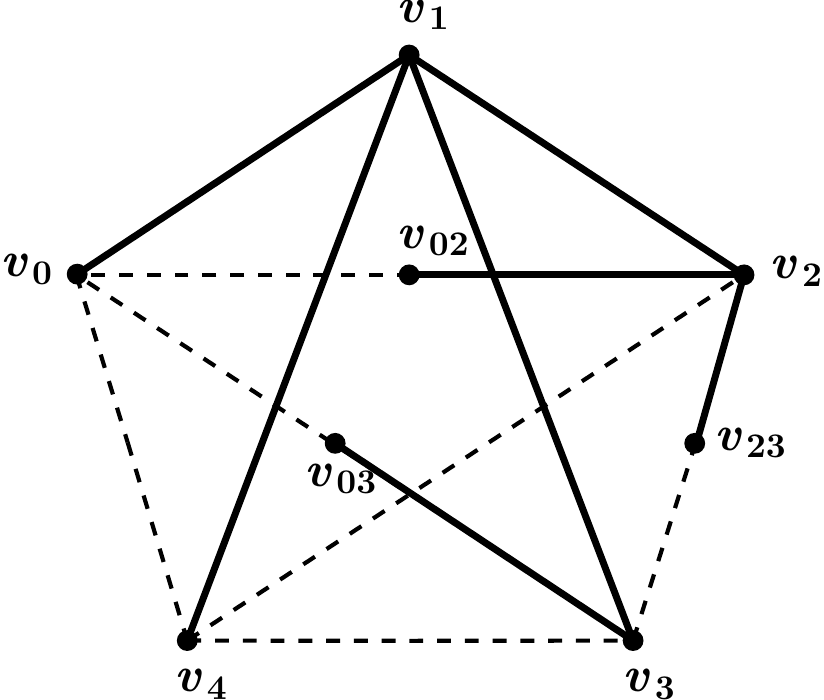}
\end{center}
\caption{Bisection of three  edges to obtain the category of the geometric realization}
\label{TRESEDGES}
\end{figure}

\end{example}

\begin{corollary}\label{GEOMGRAPH}
Let $G$ be a connected graph. Then $\scat(\sd G)=\cat \geo{G}$.
\end{corollary}

\begin{remark}
Taking into account the above results, it is interesting to point
out that, in the one-dimensional case, the difference $\scat G -
\scat \sd G = \scat G - \cat \geo{G} $ can be arbitrarily large,
as   can be checked by considering the complete graph $K_n$.
\end{remark}

\section*{Acknowledgements}
We thank J. Barmak for his valuable suggestions and M.J. Pereira-S\'aez for many useful discussions.

\bigskip

\address{
\noindent {\sc D.~Fern\'andez-Ternero}.
\\Dpto. de Geometr\'{\i}a y Topolog\'{\i}a, Universidad de Sevilla, Spain.\\}
\email{desamfer@us.es}

\medskip

\address{
\noindent {\sc E.~Mac\'ias-Virg\'os}.
\\{Dpto. de Matem\'aticas,} Universidade de San\-tia\-go de Compostela, Spain.\\}
\email{quique.macias@usc.es}

\medskip

\address{
\noindent {\sc E.~Minuz}.
\\Department of Mathematical Science, University of Copenhagen, Denmark.\\}
\email{qrz563@alumni.ku.dk}

\medskip

\address{
\noindent {\sc J.A.~Vilches}.
\\Dpto. de Geometr\'{\i}a y Topolog\'{\i}a, Universidad de Sevilla, Spain.\\}
\email{vilches@us.es}

\end{document}